\documentclass[12pt]{amsart}

\usepackage{amssymb}
\usepackage{enumerate}
\usepackage{amscd}

\newtheorem{introthm}{Theorem}

\newtheorem{theorem}{Theorem}[section]
\newtheorem{lemma}[theorem]{Lemma}
\newtheorem{corollary}[theorem]{Corollary}

\newtheorem{conjecture}[theorem]{Conjecture}

\newcommand{\n}[1]{\overline{#1}}

\newcommand{\Z}{\mathbb{Z}}

\newcommand{\Field}{\mathbb{F}}

\newcommand{\betti}[2]{\beta_{#2}^{#1}}
\newcommand{\abetti}[2]{\alpha_{#2}^{#1}}
\newcommand{\tensor}{\otimes}
\newcommand{\join}{*}

\newcommand{\sgn}{{\rm sgn}}
\newcommand{\brackom}[2]{\genfrac{[}{]}{0pt}{}{#1}{#2}}
\newcommand{\etal}{et al$.$}

\newcommand{\M}[1][{}]{\mathsf{M}_{#1}}
\newcommand{\Mupp}[1][{}]{\Gamma_{#1}}
\newcommand{\Symm}[1]{\mathfrak{S}_{#1}}
\newcommand{\zivaljevic}{$\check{\mathrm{Z}}$ivaljevi\'c}
\newcommand{\shareshian}{Shareshian}

\begin{document}

  \begin{abstract}
    Let $1 \le m \le n$.
    We prove various results about 
    the chessboard complex $\M[m,n]$, which is the simplicial complex of
    matchings in the complete bipartite graph $K_{m,n}$. 
    First, we demonstrate that there is nonvanishing
    $3$-torsion in $\tilde{H}_d(\M[m,n];\Z)$ whenever
    $\frac{m+n-4}{3} \le d \le m-4$ and whenever $6 \le m < n$ and
    $d=m-3$.
    Combining this result with theorems due to 
    Friedman and Hanlon and to {\shareshian} and Wachs,
    we characterize all triples $(m,n,d)$ satisfying
    $\tilde{H}_{d}(\M[m,n];\Z) \neq 0$.
    Second, for each $k \ge 0$, we show that there is a polynomial
    $f_k(a,b)$ of degree $3k$ such that the dimension of
    $\tilde{H}_{k+a+2b-2}(\M[k+a+3b-1,k+2a+3b-1];\Z_3)$, viewed as a
    vector space over $\Z_3$, is at most $f_k(a,b)$ for all $a \ge 0$
    and $b \ge k+2$. Third, we give a computer-free
    proof that $\tilde{H}_2(\M[5,5];\Z) \cong \Z_3$. Several 
    proofs are based on a new long exact sequence relating the
    homology of a certain subcomplex of $\M[m,n]$ to the homology of
    $\M[m-2,n-1]$ and $\M[m-2,n-3]$.    
  \end{abstract}

  \title[$3$-torsion in the Chessboard Complex]{On the $3$-torsion
    Part of the Homology of the Chessboard 
    Complex}
  \author{Jakob Jonsson}
  \thanks{Research supported by European Graduate
    Program
    ``Combinatorics, Geometry, and Computation'', DFG-GRK 588/2.}
  \address{Department of Mathematics, KTH, 10044 Stockholm, Sweden}
  \email{jakobj@math.kth.se}
  \date{}
  \keywords{matching complex, chessboard complex, simplicial
    homology}
  \subjclass{55U10, 05E25}

  \maketitle

  \noindent
  This is a preprint version of a paper published in 
  {\em Annals of Combinatorics} {\bf 14} (2010), No. 4, 487--505.

\section{Introduction}

  Given a family $\Delta$ of graphs on a fixed vertex set, we 
  identify each member of $\Delta$ with its edge set. In particular,
  if $\Delta$ is closed under deletion of edges, then 
  $\Delta$ is an abstract simplicial complex.

  A {\em matching} in a simple graph $G$ is a subset $\sigma$ of the
  edge set of $G$ such that no vertex appears in more than one edge in
  $\sigma$. Let $\M(G)$ be the family of matchings in $G$; $\M(G)$ is
  a simplicial complex. We write $\M[n] = \M(K_n)$
  and $\M[m,n] = \M(K_{m,n})$, where $K_n$ is the complete graph on
  the vertex set $[n] = \{1, \ldots, n\}$
  and $K_{m,n}$ is the complete bipartite graph with block sizes $m$
  and $n$. $\M[n]$ is the {\em matching complex} and 
  $\M[m,n]$ is the {\em chessboard complex}.

  The
  topology of $\M[n]$, $\M[m,n]$, and related complexes has been
  subject to analysis
  in a number of theses \cite{Andersen,Dong,Garst,thesis,Kara,Kson}
  and papers 
  \cite{Ath,BBLSW,BLVZ,Bouc,DongWachs,FH,KRW,RR,ShWa,Ziegvert}; see
  Wachs \cite{Wachs} for an excellent survey and further references.

  Despite the simplicity of the definition, the homology of the
  matching complex $\M[n]$
  and the chessboard complex $\M[m,n]$ 
  turns out to have a complicated structure. The rational homology 
  is well-understood and easy to describe thanks to beautiful results
  due to Bouc \cite{Bouc} and Friedman and Hanlon \cite{FH}, but very
  little is known about the integral homology and the homology over
  finite fields. 

  A previous paper \cite{bettimatch} contains a number of results
  about the integral homology of the matching complex $\M[n]$. The
  purpose of the present paper is to extend a few of these results to
  the chessboard complex $\M[m,n]$.

  For $1 \le m \le n$, define
  \[
  \nu_{m,n} = \min \{m-1, \lceil \mbox{$\frac{m+n-4}{3}$}\rceil\}
  = \left\{
  \begin{array}{ll}
    \lceil \frac{m+n-4}{3}\rceil & \mbox{if } m \le n \le 2m+1; \\
    m-1 & \mbox{if } n \ge 2m-1.
  \end{array}
  \right.
  \]
  Note that $\lceil \frac{m+n-4}{3}\rceil = m-1$ for $2m-1 \le n \le 2m+1$.
  By a theorem due to {\shareshian} and Wachs \cite{ShWa},
  $\M[m,n]$  contains nonvanishing homology in degree
  $\nu_{m,n}$ for all $m,n \ge 1$ except $(m,n) = (1,1)$. 
  Previously, Friedman and Hanlon demonstrated that this bottom
  nonvanishing homology group is finite if and only if $m \le n \le
  2m-5$ and $(m,n) \notin \{(6,6),(7,7),(8,9)\}$. 

  To settle their theorem, {\shareshian} and Wachs demonstrated that
  $\tilde{H}_{\nu_{m,n}}(\M[m,n];\Z)$ contains nonvanishing
  $3$-torsion whenever the group is finite.
  One of our main results provides upper bounds on the rank of
  the $3$-torsion part. Specifically, in
  Section~\ref{boundschess-sec}, we prove the following:
  \begin{introthm}
    For each $k \ge 0$, $a \ge 0$, and $b \ge k+2$, we have that
    $\dim \tilde{H}_{k+a+2b-2}(\M[k+a+3b-1,k+2a+3b-1];\Z_3)$ is bounded by a
    polynomial in $a$ and $b$ of degree $3k$.
    \label{boundsintro-thm}
  \end{introthm}
  An equivalent way of expressing 
  Theorem~\ref{boundsintro-thm} is to say that
  \[
  \dim \tilde{H}_{d}(\M[m,n];\Z_3) \le f_{3d-m-n+4}(n-m,m-d-1)
  \]
  whenever $m \le n \le 2m-5$ and 
  $\frac{m+n-4}{3} \le d \le \frac{2m+n-7}{4}$,
  where $f_k$ is a polynomial of degree $3k$ for each $k$. 
  The bound in Theorem~\ref{boundsintro-thm} remains true over any
  coefficient field.

  Note that we express the theorem in terms of 
  the following transformation: 
  \begin{equation}
    \left\{
    \begin{array}{ccrcrcrcl}
      k &\!\!=&\!\! -m&\!\!-&\!\!n&\!\!+&\!\! 3d&\!\!+&\!\!4
      \\
      a &\!\!=&\!\! -m&\!\!+&\!\!n
      \\
      b &\!\!=&\!\! m&\!\! &\!\! &\!\!-&\!\!d &\!\! - &\!\!1
    \end{array}
    \right. \Leftrightarrow
    \left\{
    \begin{array}{ccrcrcrcl}
      m &\!\!=&\!\! k&\!\!+&\!\!a&\!\!+&\!\!3b&\!\!-&\!\!1 \\
      n &\!\!=&\!\! k&\!\!+&\!\!2a&\!\!+&\!\!3b&\!\!-&\!\!1 \\
      d &\!\!=&\!\! k&\!\!+&\!\!a&\!\!+&\!\!2b&\!\!-&\!\!2.
    \end{array}
    \right.   
    \label{mnd2kab-eq}
  \end{equation}
  Assuming that $m \le n$, each of the three new variables
  measures the difference between two important parameters: 
  \begin{itemize}
  \item
    For $m \le n \le 2m+1$, we have that $k$ measures the difference
    between the degree $d$ and the bottom degree in which $\M[m,n]$
    has nonvanishing homology; $\frac{k}{3} = d - \frac{m+n-4}{3}$. 
  \item
    $a$ is the difference between the block sizes $n$ and $m$.
  \item 
    $b$ is the difference between $\dim \M[m,n] = m-1$ and $d$.
  \end{itemize}

  To establish Theorem~\ref{boundsintro-thm}, we 
  introduce two new long exact sequences;
  see Sections~\ref{exseq00-G-11-sec} and \ref{exseqG-21-23-sec}. 
  These two sequences involve the subcomplex $\Mupp[m,n]$ of
  $\M[m,n]$ obtained by fixing a vertex in the block of
  size $n$ and removing all but two of the 
  edges that are incident to this vertex. 
  Our first sequence is very simple and relates the homology of
  $\M[m,n]$ to that of $\Mupp[m,n]$ and $\M[m-1,n-1]$. Our second 
  sequence is more complicated and relates $\Mupp[m,n]$ to
  $\M[m-2,n-1]$ and $\M[m-2,n-3]$. 
  Combining the two sequences and ``cancelling out''
  $\Mupp[m,n]$, we obtain a bound 
  on the dimension of the $\Z_3$-homology of $\M[m,n]$ in terms of 
  $\M[m-1,n-1]$, $\M[m-2,n-1]$, and $\M[m-2,n-3]$. 
  By an induction argument, we obtain 
  Theorem~\ref{boundsintro-thm}.

  For $k=0$, Theorem~\ref{boundsintro-thm} says that 
  $\dim \tilde{H}_{\nu_{m,n}}(\M[m,n];\Z_3)$ is bounded by a constant
  whenever $m \le n \le 2m-5$ and $m+n \equiv 1 \pmod{3}$. 
  Indeed, {\shareshian} and Wachs \cite{ShWa} proved that
  $\tilde{H}_{\nu_{m,n}}(\M[m,n];\Z) \cong \Z_3$ 
  for any $m$ and $n$ satisfying these equations. 
  Their proof was by induction on $m+n$ and relied on
  a computer calculation of $\tilde{H}_2(\M[5,5];\Z)$. 
  In Section~\ref{bottomchess-sec}, we provide a
  computer-free proof that $\tilde{H}_2(\M[5,5];\Z) \cong \Z_3$,
  again using the exact sequences from Sections~\ref{exseq00-G-11-sec}
  and \ref{exseqG-21-23-sec}.

  In Section~\ref{higherchess-sec},
  we use results about the matching complex $\M[n]$ from a previous
  paper \cite{bettimatch} to extend 
  {\shareshian} and Wachs' $3$-torsion result 
  to higher-degree groups: 
  \begin{introthm}
    For $m+1 \le n \le 2m-5$, there is $3$-torsion
    in $\tilde{H}_{d}(\M[m,n]; \Z)$ whenever
    $\frac{m+n-4}{3} \le d \le m-3$.
    There is also $3$-torsion
    in $\tilde{H}_{d}(\M[m,m]; \Z)$ whenever
    $\frac{2m-4}{3} \le d \le m-4$.
    \label{nonvan3intro-thm}
  \end{introthm}
  Note that $m+1 \le n \le 2m-5$ and $\frac{m+n-4}{3} \le d \le m-3$
  if and only if $k\ge 0$, $a \ge 1$, and $b \ge 2$, 
  where $k$, $a$, and $b$ are defined as in (\ref{mnd2kab-eq}).
  Moreover, $m=n$ and $\frac{2m-4}{3} \le d \le m-4$
  if and only if $k\ge 0$, $a = 0$, and $b \ge 3$.

  Our proof of Theorem~\ref{nonvan3intro-thm} relies on properties of
  the top homology group of $\M[k,k+1]$ for different values of $k$;
  this group was of importance also in the work of {\shareshian} and
  Wachs \cite{ShWa}. 

  Thanks to Theorem~\ref{nonvan3intro-thm} and Friedman and Hanlon's
  formula for the rational homology \cite{FH}, we may
  characterize those $(d,m,n)$ satisfying $\tilde{H}_d(\M[m,n];\Z)
  \neq 0$:
  \begin{introthm}
    For $1 \le m \le n$, we have that $\tilde{H}_d(\M[m,n];\Z)$ is
    nonzero if and only if either of the following is true:
    \begin{itemize}
    \item
      $\lceil\frac{m+n-4}{3}\rceil \le d \le m-2$. Equivalently, $k\ge
      0$, $a \ge 0$, and $b \ge 1$.
    \item
      $d= m-1$ and $n \ge m+1$. Equivalently, $k\ge 2-a$, $a \ge
      1$, and $b = 0$.
    \end{itemize}
  \end{introthm}
  Again, see Section~\ref{higherchess-sec} for details.

  The problem of detecting $p$-torsion in the homology of
  $\tilde{M}_{m,n}$ for $p \neq 3$ remains
  open. In this context, we may mention that there is $5$-torsion in
  the homology of the matching complex $\M[14]$ \cite{m14}. 
  By computer calculations \cite{moretorsion}, further 
  $p$-torsion is known to exist for $p \in \{5, 7, 11, 13\}$.

  \subsection{Notation}
  \label{basic-sec}

  We identify the two parts of the graph $K_{m,n}$ with the two sets
  $[m] = \{1, 2, \ldots, m\}$ and $[\n{n}] = \{\n{1}, \n{2}, \ldots,
  \n{n}\}$. The latter set should be interpreted as a disjoint copy of
  $[n]$; hence each
  edge is of the form $i\n{j}$, where $i \in [m]$ and $j \in [n]$.
  Sometimes, it will be useful to view $\M[m,n]$ as a subcomplex of
  the matching complex $\M[m+n]$ on the complete graph $K_{m+n}$. In
  such situations, we identify the vertex $\n{j}$ in $K_{m,n}$ with
  the vertex $m+j$ in $K_{m+n}$ for each $j \in [n]$.

  For finite sets $S$ and $T$, we let $\M[S,T]$ denote the matching
  complex on the complete bipartite graph with blocks $S$ and $T$,
  viewed as disjoint sets in the manner described above.
  In particular, $\M[{[m],[n]}] = \M[m,n]$.
  For integers $a \le b$, we write $[a,b] = \{a, a+1, \ldots, b-1, b\}$.

  The {\em join} of two families of sets $\Delta$ and $\Sigma$,
  assumed to be defined on disjoint ground sets, is the family
  $\Delta \join \Sigma = \{ \delta \cup \sigma: \delta \in \Delta,
  \sigma \in \Sigma\}$.

  Whenever we discuss the homology of a simplicial complex or the
  relative homology of a pair of simplicial complexes, we mean reduced
  simplicial homology. 
  For a simplicial complex $\Sigma$ and a coefficient ring
  $\Field$, we let $e_0 \wedge \cdots \wedge e_d$ denote a generator
  of $\tilde{C}_d(\Sigma; \Field)$ corresponding to the set $\{e_0,
  \ldots, e_d\} \in \Sigma$. 
  Given a cycle $z$ in a chain group $\tilde{C}_d(\Sigma;
  \Field)$, whenever we talk about $z$ as an element in the 
  induced homology group $\tilde{H}_d(\Sigma;\Field)$, we really mean
  the homology class of $z$.

  We will often consider pairs of
  complexes $(\Gamma, \Delta)$ such that
  $\Gamma \setminus \Delta$ is a union of families of the
  form
  \[
  \Sigma = \{\sigma\} \join \M[S,T], 
  \]
  where $\sigma = \{e_1, \ldots, e_s\}$ is a set of pairwise disjoint
  edges of the form $i\n{j}$, and where $S$ and $T$
  are subsets of $[m]$ and $[n]$, respectively, such that $S \cap
  e_i = \n{T} \cap e_i = \emptyset$ for each $i$.
  We may write the chain complex of $\Sigma$ as
  \[
  \tilde{C}_d(\Sigma;\Field) = (e_1 \wedge \cdots \wedge e_s)\Field
  \tensor_\Field \tilde{C}_{d-|\sigma|}(\M[S,T];\Field),
  \]
  defining the boundary operator as 
  \[
  \partial(e_1 \wedge \cdots \wedge e_s \tensor_\Field c)
  = (-1)^s e_1 \wedge \cdots \wedge e_s \tensor_\Field \partial(c).
  \]
  For simplicity, we will often suppress $\Field$ from notation. For
  example, by some abuse of notation, we will write
  \[
  e_1 \wedge \cdots \wedge e_s \tensor
  \tilde{C}_{d-|\sigma|}(\M[S,T]) =
  (e_1 \wedge \cdots \wedge e_s)\Field \tensor_\Field
  \tilde{C}_{d-|\sigma|}(\M[S,T];\Field).
  \]

  We say that a cycle $z$ in 
  $\tilde{C}_{d-1}(\M[m,n];\Field)$ has {\em type}
  $\brackom{m_1,n_1}{d_1} \wedge \cdots \wedge
  \brackom{m_s,n_s}{d_s}$ 
  if there are partitions  
  $[m] = \bigcup_{i=1}^s S_i$ and
  $[n] = \bigcup_{i=1}^s T_i$
  such that $|S_i| = m_i$ and $|T_i| = n_i$
  and such that
  $z = z_1 \wedge \cdots \wedge z_s$,
  where $z_i$ is a cycle in $\tilde{C}_{d_i-1}(\M[S_i,T_i];\Field)$ for
  each $i$. 

  \subsection{Two classical results}
  \label{classic}

  Before proceeding, we list two classical results pertaining to 
  the topology of the chessboard complex $\M[m,n]$.

  \begin{theorem}[Bj\"{o}rner
    {\etal} \cite{BLVZ}]
    For $m,n \ge 1$, $\M[m,n]$ is $(\nu_{m,n}-1)$-connected. 
    \label{smallest-thm}
  \end{theorem}
  Indeed, the $\nu_{m,n}$-skeleton of $\M[m,n]$ is 
  vertex decomposable \cite{Ziegvert}. 
  Garst \cite{Garst} settled the case $n \ge 2m-1$ in
  Theorem~\ref{smallest-thm}.
  As already mentioned in the
  introduction, there is nonvanishing homology in 
  degree $\nu_{m,n}$ for all $(m,n) \neq (1,1)$; see
  Section~\ref{bottomchess-sec} for details.

  The transformation (\ref{mnd2kab-eq}) maps the set
  $\{(m,n,\nu_{m,n}) : 1 \le m \le n\}$ to
  the set of triples $(k,a,b)$ satisfying either of the following:
  \begin{itemize}
  \item
    $k \in \{0,1,2\}$, $a \ge 0$, and $b \ge 1$ (corresponding to 
    $d = \lceil \frac{m+n-4}{3} \rceil$ and $m \le n \le 2m-2$).
  \item
    $2-a \le k \le 2$ and $b = 0$ (corresponding to $0 \le d = m-1$
    and $n \ge 2m-1$).
  \end{itemize}

  Friedman and Hanlon \cite{FH} established a formula for 
  the rational homology of $\M[m,n]$; see Wachs \cite{Wachs}
  for an overview. For our purposes, the most important consequence is
  the following result:

  \begin{theorem}[Friedman and Hanlon \cite{FH}]
    For $1 \le m \le n$, we have that
    $\tilde{H}_{d}(\M[m,n]; \Z)$ is infinite if and only if
    $(m-d-1)(n-d-1) \le d+1$, $m \ge d+1$, and $n \ge d+2$. 
    In particular, $\tilde{H}_{\nu_{m,n}}(\M[m,n]; \Z)$ is finite
    if and only if $n \le 2m-5$ and $(m,n) \notin
    \{(6,6),(7,7),(8,9)\}$.
    \label{chessfinite-thm}
  \end{theorem}
  With $k$, $a$, and $b$ defined as in (\ref{mnd2kab-eq}), 
  the conditions
  $1 \le m \le n$, $(m-d-1)(n-d-1) \le d+1 \le m$, and $n \ge d+2$
  are equivalent to 
  \[
  b(a+b) \le k + a + 2b - 1 \Longleftrightarrow
  (b-1)(a+b-1) \le k, 
  \]
  $a \ge 0$, $b \ge 0$, $a+b \ge 1$, and $k+a+3b \ge 2$.
  Moreover, the conditions $d = \nu_{m,n}$, $m \le n \le 2m-5$, 
  and $(m,n) \notin \{(6,6),(7,7),(8,9)\}$ are equivalent to 
  $k \in \{0,1,2\}$, $a \ge 0$, $b \ge 2$, and $(k,a,b) \notin
  \{(1,0,2), (2, 0, 2), (2, 1, 2)\}$.
  
  \section{Four long exact sequences}

  We present four long exact sequences relating different
  families of chessboard complexes. In this paper, we will only use
  the third and the fourth sequences; we list the other two sequences 
  for reference. Throughout this section,
  we consider an arbitrary coefficient ring $\Field$, which we suppress
  from notation for convenience.

  \subsection{Long exact sequence relating $\M[m,n]$,
    $\M[m,n-1]$, and $\M[m-1,n-1]$}
  \label{exseq00-01-11-sec}

  \begin{theorem}
    Define
    \[
    P^{m-1,n-1}_{d} = 
    \bigoplus_{s=1}^m s\n{1} \tensor\tilde{H}_{d}(\M[{[m]\setminus
        \{s\},[2,n]}]).
    \]
    For each $m \ge 1$ and $n \ge 1$, we have a long exact sequence
    \[
    \begin{CD}
      & & & & \cdots @>>> 
      P^{m-1,n-1}_{d} \\
      @>>> 
      \tilde{H}_{d}(\M[m,n-1]) @>>>
      \tilde{H}_{d}(\M[m,n]) @>>>
      P^{m-1,n-1}_{d-1} \\
      @>>>
      \tilde{H}_{d-1}(\M[m,n-1]) @>>> \cdots .
    \end{CD}
    \]
  \end{theorem}
  \begin{proof}
    This is the long exact sequence for the 
    pair $(\M[m,n], \M[m,n-1])$.
  \end{proof}
  We refer to this sequence as the {\em 00-01-11 sequence}, thereby
  indicating that the sequence relates $\M[m-0,n-0]$, $\M[m-0,n-1]$,
  and $\M[m-1,n-1]$. 
  Note that the sequence is asymmetric in $m$ and $n$; 
  swapping the indices, we obtain an exact sequence
  relating 
  $\M[m,n]$, $\M[m-1,n]$, and $\M[m-1,n-1]$.

  \subsection{Long exact sequence relating $\M[m,n]$,
    $\M[m-1,n-2]$, $\M[m-2,n-1]$, and $\M[m-2,n-2]$}
  \label{exseq00-12-21-22-sec}

  \begin{theorem}[{\shareshian} \& Wachs \cite{ShWa}]
    Define
    \begin{eqnarray*}
    P^{m-1,n-2}_{d} &=& 
    \bigoplus_{t=2}^n 1\n{t} \tensor\tilde{H}_{d}(\M[{[2,m],[2,n]\setminus
        \{t\}}]); \\
    Q^{m-2,n-1}_{d} &=& 
    \bigoplus_{s=2}^m s\n{1}
    \tensor\tilde{H}_{d}(\M[{[2,m]\setminus\{s\},[2,n]}]); \\
    R^{m-2,n-2}_{d} &=&
    \bigoplus_{s=2}^m \bigoplus_{t=2}^n
    1\n{t} \wedge s\n{1} \tensor
    \tilde{H}_{d}(\M[{[2,m] \setminus \{s\},[2,n] \setminus \{t\}}]).
    \end{eqnarray*}
    For each $m \ge 2$ and $n \ge 2$, we have a long exact sequence
    \[
    \begin{CD}
      & & & & \cdots @>>> 
      R^{m-2,n-2}_{d-1} \\
      @>>> 
      P^{m-1,n-2}_{d-1}
      \oplus 
      Q^{m-2,n-1}_{d-1}
      @>>>
      \tilde{H}_{d}(\M[m,n]) @>>>
      R^{m-2,n-2}_{d-2} \\
      @>>>
      P^{m-1,n-2}_{d-2}
      \oplus 
      Q^{m-2,n-1}_{d-2} @>>> \cdots .
    \end{CD}
    \]
  \end{theorem}
  We refer to this sequence as the {\em 00-12-21-22 sequence}.
  The sequence played an important part in {\shareshian} and Wachs'
  analysis \cite{ShWa} of the bottom nonvanishing homology of
  $\M[m,n]$.
  Note that the sequence is symmetric in $m$ and $n$.

  \subsection{Long exact sequence relating $\M[m,n]$,
    $\Mupp[m,n]$, and $\M[m-1,n-1]$}
  \label{exseq00-G-11-sec}

  The sequence in this section is very similar, but not identical, to
  the 00-01-11 sequence in Section~\ref{exseq00-01-11-sec}.
  Define 
  \begin{equation}
    \label{Gamma-eq}
    \Mupp[m,n] = \{\sigma \in \M[m,n]: s\n{1} \notin \sigma \mbox{
    for } s \in [3,m]\}.
  \end{equation}

  \begin{theorem}
    Define
    \[
    \hat{P}^{m-1,n-1}_{d} = 
    \bigoplus_{s=3}^m s\n{1} \tensor\tilde{H}_{d}(\M[{[m]\setminus
        \{s\},[2,n]}]);
    \]
    note that this definition differs from that in 
    Section~{\rm\ref{exseq00-01-11-sec}}.
    For each $m \ge 1$ and $n \ge 1$, we have a long exact sequence
    \[
    \begin{CD}
      & & & & \cdots @>>> 
      \hat{P}^{m-1,n-1}_{d} \\
      @>>> 
      \tilde{H}_{d}(\Mupp[m,n]) @>>>
      \tilde{H}_{d}(\M[m,n]) @>>>
      \hat{P}^{m-1,n-1}_{d-1} \\
      @>>>
      \tilde{H}_{d-1}(\Mupp[m,n]) @>>> \cdots .
    \end{CD}
    \]
  \end{theorem}
  \begin{proof}
    This is the long exact sequence for the pair 
    $(\M[m,n],\Mupp[m,n])$.
  \end{proof}
  We refer to this sequence as the {\em 00-$\mathit{\Gamma}$-11
  sequence}.
  Note that the sequence is asymmetric in $m$ and $n$.

  \subsection{Long exact sequence relating $\Mupp[m,n]$,
    $\M[m-2,n-1]$, and $\M[m-2,n-3]$}
  \label{exseqG-21-23-sec}

  Recall the definition of $\Mupp[m,n]$ from 
  (\ref{Gamma-eq}).
  \begin{theorem}
    Write
    \begin{eqnarray*}
      Q^{m-2,n-1}_{d} &=&
      (1\n{1} - 2\n{1}) \tensor\tilde{H}_{d}(\M[{[3,m],[2,n]}]); \\
      R^{m-2,n-3}_{d} &=&
      \bigoplus_{s \neq t \in [2,n]}
      1\n{s}\wedge 2\n{t} \tensor
      \tilde{H}_{d}(\M[{[3,m],[2,n] \setminus \{s,t\}}]).
    \end{eqnarray*}
    For each $m \ge 2$ and $n \ge 3$, we have a long exact sequence
    \[
    \begin{CD}
      & & & & \cdots @>>> 
      R^{m-2,n-3}_{d-1} \\
      @>\varphi^*>> 
      Q^{m-2,n-1}_{d-1} @>\iota^*>>
      \tilde{H}_{d}(\Mupp[m,n]) @>>>
      R^{m-2,n-3}_{d-2} \\
      @>>>
      Q^{m-2,n-1}_{d-2} @>>> \cdots ,
    \end{CD}
    \]
    where $\varphi^*$ is induced by the map 
    $\varphi$ defined by 
    \[
    \varphi(1\n{s} \wedge 2\n{t}\tensor x) 
    = (1\n{1} - 2\n{1}) \tensor x.
    \]
    and $\iota^*$ is induced by the natural map
    $\iota((1\n{1} - 2\n{1}) \tensor x) = (1\n{1} - 2\n{1}) \wedge
    x$. 
  \label{exseqG-21-23-thm}
  \end{theorem}
  \begin{proof}
    Define a filtration 
    \[
    \Delta^0_{m,n} \subset
    \Delta^1_{m,n} \subset
    \Delta^2_{m,n} = \Mupp[m,n]
    \]
    as follows:
    \begin{itemize}
    \item
      $\Delta^2_{m,n} = \Mupp[m,n]$.
    \item
      $\Delta^1_{m,n}$ is the subcomplex of $\Delta^2_{m,n}$ obtained by 
      removing all faces containing $\{1\n{s},2\n{t}\}$ for some
      $s,t \in [2,n]$.
    \item
      $\Delta^0_{m,n}$ is the subcomplex of $\Delta^1_{m,n}$ obtained by 
      removing the elements $1\n{2}, \ldots, 1\n{n}$ 
      and $2\n{2}, \ldots, 2\n{n}$.
    \end{itemize}
    Writing $\Delta^{-1}_{m,n} = \emptyset$, let us 
    examine $\Delta^{i}_{m,n} \setminus
    \Delta^{i-1}_{m,n}$ for $i = 0,1,2$.

    $\bullet$ $i=0$. Note that
    \[
    \Delta^0_{m,n} = \M[2,1] \join \M[{[3,m],[2,n]}]
    \cong
    \M[2,1] \join \M[m-2,n-1].
    \]
    As a consequence, 
    \[
    \tilde{H}_d(\Delta^0_{m,n})
    \cong (1\n{1} -2\n{1}) \tensor 
    \tilde{H}_{d-1}(\M[{[3,m],[2,n]}])
    = Q_{d-1}^{m-2,n-1}.
    \]

    $\bullet$ $i=1$.    
    Observe that
    \[
    \Delta^1_{m,n} \setminus \Delta^0_{m,n}
    = \bigcup_{a=1}^2\bigcup_{u=2}^n \{\{a\n{u}\}\} 
    \join \M[\{3-a\},\{1\}] \join \M[{[3,m],[2,n] \setminus
    \{u\}}].
    \]
    It follows that
    \[
    \tilde{H}_d(\Delta^1_{m,n},\Delta^0_{m,n}) = 
    \bigoplus_{a,u}
    a\n{u} \tensor \tilde{H}_{d-1}(\M[\{3-a\},\{1\}] \join
    \M[{[3,m],[2,n] \setminus
        \{u\}}]) = 0;
    \]
    $\M[\{3-a\},\{1\}] \cong \M[1,1]$ is a point.
    In particular, $\tilde{H}_d(\Delta^1_{m,n}) \cong
    \tilde{H}_d(\Delta^0_{m,n})$.

    $\bullet$ $i=2$.    
    We have that
    \[
    \Delta^2_{m,n} \setminus \Delta^1_{m,n}
    = 
    \bigcup_{s,t \in [2,n]} \{\{1\n{s},2\n{t}\}\} \join \M[{[3,m],[2,n] \setminus
    \{s,t\}}];
    \]
    we may hence conclude that 
    \[
    \tilde{H}_d(\Delta^2_{m,n},\Delta^1_{m,n}) = 
    \bigoplus_{s,t} 1\n{s} \wedge 2\n{t} \tensor 
    \tilde{H}_{d-2}(\M[{[3,m],[2,n] \setminus
        \{s,t\}}]) = R_{d-1}^{m-2,n-3}.
    \]
    By the long exact sequence for the pair 
    $(\Delta^2_{m,n},\Delta^1_{m,n})$, it remains to prove that the 
    induced map
    $\varphi^*$ has properties as stated in the theorem.
    Now, in the long exact sequence for $(\Delta^2_{m,n},\Delta^1_{m,n})$, 
    the induced boundary map from
    $\tilde{H}_{d+1}(\Delta^2_{m,n},\Delta^1_{m,n})$ to 
    $\tilde{H}_{d}(\Delta^1_{m,n})$ maps the element 
    $1\n{s} \wedge 2\n{t}\tensor z$
    to $(2\n{t} - 1\n{s}) \tensor z$. 
    Since
    \[
    (2\n{t} - 1\n{s}) \tensor z-\partial((1\n{1} \wedge 2\n{t}
    + 1\n{s}\wedge 2\n{1}) 
    \tensor z)  = (1\n{1} - 2\n{1}) \tensor z,
    \]
    we are done.
  \end{proof}
  We refer to the sequence in Theorem~\ref{exseqG-21-23-thm} as the
  {\em $\mathit{\Gamma}$-21-23 sequence}. Note that the sequence is
  asymmetric in $m$ and $n$.

  \section{Bottom nonvanishing homology}
  \label{bottomchess-sec}

  Using the long exact sequences in Sections~\ref{exseq00-G-11-sec}
  and \ref{exseqG-21-23-sec}, we give a computer-free proof
  that $\tilde{H}_{2}(\M[5,5];\Z)$ is a group of size three.  
  While the proof is complicated, our hope is that it may provide at
  least some insight into the structure of $\M[5,5]$ and related 
  chessboard complexes.
  \begin{theorem}
    We have that $\tilde{H}_{2}(\M[5,5];\Z) \cong \Z_3$.
    \label{m55-thm}
  \end{theorem}
  \begin{proof}
    First, we examine $\M[3,4]$; for alignment with later parts of the
    proof, we consider $\M[{[3,5],[2,5]}]$, thereby shifting the first
    index two steps and the second index one step.
    The long exact 00-$\Gamma$-11 sequence
    from Section~\ref{exseq00-G-11-sec} becomes
    \[
    \begin{CD}
       0 @>>>
      \tilde{H}_2(\Mupp[{[3,5],[2,5]}]) @>>> 
      \tilde{H}_2(\M[{[3,5],[2,5]}]) @>\omega^*>>
      5\n{2}\tensor\tilde{H}_1(\M[{[3,4],[3,5]}]) \\
      @>>>
      \tilde{H}_1(\Mupp[{[3,5],[2,5]}]) @>\iota^*>> 
      \tilde{H}_1(\M[{[3,5],[2,5]}]) @>>>
      0.
    \end{CD}
    \]
    As {\shareshian} and Wachs observed \cite[\S6]{ShWa},
    the complex $\M[m,m+1]$ is an orientable pseudomanifold of
    dimension $m-1$. In particular, 
    $\M[{[3,5],[2,5]}]$ and
    $\M[{[3,4],[3,5]}]$ are
    orientable pseudomanifolds of dimensions $2$ and $1$,
    respectively. Moreover, the top homology group 
    of $\M[{[3,5],[2,5]}]$ is generated by
    \[
    z = \sum_{\pi \in \Symm{[2,5]}}
    \sgn(\pi) \cdot 3\n{\pi(3)}\wedge 4\n{\pi(4)} \wedge
    5\n{\pi(5)},
    \]
    and the top homology group of $\M[{[3,4],[3,5]}]$ is generated by
    \[
    z' = \sum_{\pi \in \Symm{[3,5]}}
    \sgn(\pi) \cdot 3\n{\pi(3)}\wedge 4\n{\pi(4)}.
    \]
    Since $\omega^*(z) = -z'$, the map $\omega^*$ is an isomorphism.
    As a consequence, the map $\iota^*$ induced by the
    natural inclusion map is also an isomorphism. 

    The long exact $\Gamma$-21-23 sequence for $\Mupp[{[3,5],[2,5]}]$
    from Section~\ref{exseqG-21-23-sec} becomes
    \[
    \begin{CD}
       0 @>>>
      (3\n{2} - 4\n{2})\tensor\tilde{H}_0(\M[{\{5\},[3,5]}])
      @>\iota^*>>
      \tilde{H}_1(\Mupp[{[3,5],[2,5]}]) @>>> 
      0,
    \end{CD}
    \]
    which yields that each of $\tilde{H}_1(\Mupp[{[3,5],[2,5]}])$ and 
    $\tilde{H}_1(\M[{[3,5],[2,5]}])$ is generated
    by $e_i = (3\n{2} - 4\n{2}) \wedge (5\n{3}-5\n{i})$
    for $i \in \{4,5\}$.

    Now, consider $\M[5,5]$. The tail end 
    of the $\Gamma$-21-23 sequence is
    \[
    \begin{CD}
      & & \displaystyle{\bigoplus_{s,t}}\ 1\n{s} \wedge 2\n{t}\tensor
      \tilde{H}_1(\M[{[3,5],[2,5]\setminus \{s,t\}}]) \\
      @>\varphi^*>>
      (1\n{1}-2\n{1})\tensor \tilde{H}_1(\M[{[3,5],[2,5]}]) 
      @>\iota^*>> \tilde{H}_2(\Mupp[5,5]) \rightarrow
      0,
    \end{CD}
    \]
    where the first sum ranges over all pairs of distinct elements
    $s,t \in [2,5]$.
    Writing $\{s,t,u,v\} = [2,5]$,
    we note that $\tilde{H}_1(\M[{[3,5],[2,5]\setminus
    \{s,t\}}]) = \tilde{H}_1(\M[{[3,5],\{u,v\}}])$ is generated by the cycle
    \[
    z_{uv} = 3\n{u} \wedge 4\n{v} + 4\n{v} \wedge 5\n{u} + 5\n{u}
    \wedge 3\n{v} + 3\n{v} \wedge 4\n{u} + 4\n{u} \wedge 5\n{v} +
    5\n{v} \wedge 3\n{u}. 
    \]
    By Theorem~\ref{exseqG-21-23-thm}, $\varphi^*$ maps 
    $1\n{s} \wedge 2\n{t}\tensor z_{uv}$ to
    $(1\n{1} - 2\n{1}) \tensor z_{uv}$. 
    Since $z_{uv} = z_{vu}$, we conclude that
    the image under $\varphi^*$ is 
    generated by the six cycles
    $z_{23},z_{24},z_{25}, z_{34},z_{35},z_{45}$.

    In $\tilde{H}_1(\M[{[3,5],[2,5]}])$, we have that  
    $z_{st} = z_{uv}$, because
    $z_{st} - z_{uv}$ equals the boundary of 
    \begin{eqnarray*}
      \gamma &=& 3\n{u} \wedge 5\n{s} \wedge 4\n{v} 
      - 
      5\n{s} \wedge 4\n{v} \wedge 3\n{t}
      + 
      4\n{v} \wedge 3\n{t} \wedge 5\n{u}
      - 
      3\n{t} \wedge 5\n{u} \wedge 4\n{s} \\
      &+& 
      5\n{u} \wedge 4\n{s} \wedge 3\n{v}
      - 
      4\n{s} \wedge 3\n{v} \wedge 5\n{t}
      + 
      3\n{v} \wedge 5\n{t} \wedge 4\n{u} 
      -
      5\n{t} \wedge 4\n{u} \wedge 3\n{s} \\
      &+& 
      4\n{u} \wedge 3\n{s} \wedge 5\n{v}
      - 
      3\n{s} \wedge 5\n{v} \wedge 4\n{t}
      + 
      5\n{v} \wedge 4\n{t} \wedge 3\n{u}
      - 
      4\n{t} \wedge 3\n{u} \wedge 5\n{s}.
    \end{eqnarray*}
    Namely, $\gamma$ is of the form $a_1 \wedge a_2
    \wedge a_3 - a_2
    \wedge a_3 \wedge a_4 + \cdots - a_{12} \wedge a_1 \wedge a_2$, 
    which yields the boundary $-a_1\wedge a_3 + a_2 \wedge a_4 -
    \cdots + a_{12} \wedge a_2$. 
    As a consequence, the image under $\varphi^*$ is
    generated by the three cycles
    $z_{34},z_{35},z_{45}$.

    Assume that $s = 2$ and $\{t,u,v\} = \{3,4,5\}$ and write
    \begin{eqnarray*}
    w_{uv} &=& 
    5\n{u} \wedge 4\n{s} \wedge 3\n{v} 
    - 4\n{s} \wedge 3\n{v} \wedge 5\n{t}
    + 3\n{v} \wedge 5\n{t} \wedge 4\n{u} \\
    &-& 5\n{t} \wedge 4\n{u} \wedge 3\n{s}
    + 4\n{u} \wedge 3\n{s} \wedge 5\n{v}.
    \end{eqnarray*}
    We obtain that 
    \begin{eqnarray*}
      \partial(w_{uv}+w_{vu}) &=& 
      (5\n{u} \wedge 4\n{s} - 5\n{u} \wedge 3\n{v} + 4\n{s} \wedge
      5\n{t} 
      - 3\n{v} \wedge 4\n{u} + 5\n{t} \wedge 3\n{s} \\
      &-& 4\n{u} \wedge 5\n{v} +  3\n{s} \wedge 5\n{v}) +
      (5\n{v} \wedge 4\n{s} - 5\n{v} \wedge 3\n{u} 
      + 4\n{s} \wedge 5\n{t}\\
      &-& 3\n{u}  \wedge 4\n{v} + 5\n{t} \wedge 3\n{s}
      - 4\n{v} \wedge 5\n{u} + 3\n{s} \wedge 5\n{u})\\
      &=& 
      (4\n{s} - 3\n{s}) \wedge (2\cdot 5\n{t} - 5\n{u} - 5\n{v})
      - z_{uv}.
    \end{eqnarray*}
    Since $s=2$, it follows that $z_{uv}$ is equal to either $-e_4 -
    e_5$, $2e_4-e_5$, or $-e_4+2e_5$ in
    $\tilde{H}_1(\M[{[3,5][2,5]}])$ depending on the values of $t$, 
    $u$, and $v$.

    We conclude that the set $\{\varphi^*(1\n{s}\wedge 2\n{t}
    \tensor z_{uv}) : \{s,t,u,v\} = [2,5]\}$ 
    generates the subgroup $\{ (1\n{1}-2\n{1}) \tensor (ae_4+be_5) :
    a-b \equiv 0 \pmod{3}\}$ of 
    $(1\n{1}-2\n{1}) \tensor \tilde{H}_1(\M[{[3,5],[2,5]}])$.
    As a consequence, 
    $\tilde{H}_2(\Mupp[5,5]) \cong \Z_3$, and 
    \[
    \rho = (1\n{1}-2\n{1}) \wedge (3\n{2}-4\n{2}) \wedge
    (5\n{3}-5\n{4})
    \]
    is a generator for this group.
    Swapping $\n{3}$ and $\n{4}$, we obtain $-\rho$; 
    we obtain the same result if we swap $3$ and $4$ or if we swap 
    $1$ and $2$. Hence, by symmetry, the group 
    \[
    T = \Symm{\{1,2\}} \times \Symm{\{3,4,5\}} \times
    \Symm{\{\n{2},\n{3},\n{4},\n{5}\}}
    \]
    acts on $\tilde{H}_2(\Mupp[5,5]) \cong \Z_3$ by
    $\pi(\rho) = \sgn(\pi) \cdot \rho$.

    It remains to prove that 
    $\tilde{H}_2(\Mupp[5,5]) \cong 
    \tilde{H}_2(\M[5,5])$. For this, consider
    the  tail end of the 00-$\Gamma$-11 sequence
    from Section~\ref{exseq00-G-11-sec}:
    \[
    \begin{CD}
      \displaystyle{\bigoplus_{x=3}^5}\ x\n{1}\tensor
      \tilde{H}_2(\M[{[5]\setminus \{x\},[2,5]}]) 
      @>\psi^*>>
      \tilde{H}_2(\Mupp[5,5]) @>>> 
      \tilde{H}_2(\M[5,5]) \rightarrow
      0
    \end{CD}
    \]
    By a result due to {\shareshian} and Wachs \cite[Lemma 5.9]{ShWa},
    we have that 
    $\tilde{H}_2(\M[{[5]\setminus \{x\},[2,5]}]) \cong 
    \tilde{H}_2(\M[4,4])$ is generated by cycles 
    of type $\brackom{3,2}{2} \wedge \brackom{1,2}{1}$
    and cycles of type 
    $\brackom{2,3}{2} \wedge \brackom{2,1}{1}$; recall 
    notation from Section~\ref{basic-sec}.
    By properties of $\psi^*$, we need only prove that
    any such cycle vanishes in $\tilde{H}_2(\Mupp[5,5])$
    whenever $x \in [3,5]$.

    $\bullet$ A cycle of the first type is of the form
    $z = \lambda \cdot \gamma \wedge (d\n{u}-d\n{v})$, where
    $\lambda$ is a constant scalar,
    \[
    \gamma = a\n{s} \wedge b\n{t} + b\n{t} \wedge c\n{s} + c\n{s}
    \wedge a\n{t} + a\n{t} \wedge b\n{s} + b\n{s} \wedge c\n{t} +
    c\n{t} \wedge a\n{s},
    \]
    $\{a,b,c,d\} = [5] \setminus \{x\}$, and $\{s,t,u,v\} = [2,5]$. 
    By the above discussion, swapping $\n{s}$ and $\n{t}$ in $z$
    should yield $-z$, but obviously the same swap in $\gamma$ again
    yields $\gamma$, which implies that $z = - z$; hence $z=0$.

    $\bullet$ A cycle of the second type is of the form
    $z = \lambda \cdot \gamma \wedge (c\n{v}-d\n{v})$, where
    $\lambda$ is a constant scalar, say $\lambda = 1$, and 
    \[
    \gamma = a\n{s} \wedge b\n{t} + b\n{t} \wedge a\n{u} + a\n{u}
    \wedge b\n{s} + b\n{s} \wedge a\n{t} + a\n{t} \wedge b\n{u} +
    b\n{u} \wedge a\n{s};
    \]
    again $\{a,b,c,d\} = [5] \setminus \{x\}$ and $\{s,t,u,v\} = [2,5]$. 
    If $\{a,b\} \subset [3,5]$, then we may swap $a$ and $b$ and again
    conclude that $z = -z$; the same argument applies if $\{a,b\} =
    \{1,2\}$.
    For the remaining case, we may assume that
    $c \in [1,2]$ and $d \in [3,5]$. Swapping $d$ and $x$ yields
    $-z = \gamma \wedge (c\n{v}-x\n{v})$; recall that $x \in
    [3,5]$. As a consequence,
    \[
    2z = z-(-z) = \gamma \wedge (x\n{v}-d\n{v}) = 
    \partial(c\n{1} \wedge \gamma \wedge (x\n{v}-d\n{v}));
    \]
    hence $z$ is again zero. Namely, since $c \in [1,2]$, we have that
    $c\n{1}$ is an element in $\Mupp[5,5]$. As a consequence, 
    $\psi^*$ is the zero map as desired.
  \end{proof}

  By Theorems~\ref{smallest-thm} and
  \ref{chessfinite-thm},
  the connectivity degree of $\M[m,n]$ is exactly $\nu_{m,n}-1$
  whenever $n \ge 2m-4$ or $(m,n) \in
  \{(6,6),(7,7),(8,9)\}$.
  As mentioned in the introduction, {\shareshian} and Wachs
  \cite{ShWa} extended this result to all $(m,n) \neq (1,1)$, thereby
  settling a conjecture due to Bj\"orner {\etal} \cite{BLVZ}:
  \begin{theorem}[{\shareshian} \& Wachs \cite{ShWa}]
    If $m \le n \le 2m-5$ and $(m,n) \neq (8,9)$, then there is
    nonvanishing $3$-torsion in
    $\tilde{H}_{\nu_{m,n}}(\M[m,n]; \Z)$.
    If in addition $m+n \equiv 1 \pmod{3}$, then
    $\tilde{H}_{\nu_{m,n}}(\M[m,n]; \Z) \cong \Z_3$.
    \label{chesstorsion-thm}
  \end{theorem}
  By Theorem~\ref{chessallover-thm} in
  Section~\ref{higherchess-sec}, there is nonvanshing
  $3$-torsion also in $\tilde{H}_{\nu_{8,9}}(\M[8,9]; \Z)$; in that
  theorem, choose $(k,a,b) = (2,1,2)$.

  \bigskip

  \noindent
  {\bf[Table 1]}

  \bigskip

  In fact, {\shareshian} and Wachs provided much more specific
  information about the exponent of $\tilde{H}_{\nu_{m,n}}(\M[m,n];
  \Z)$; see Table~\ref{chessexp-fig}.

  \begin{conjecture}[{\shareshian} \& Wachs \cite{ShWa}]
    The group $\tilde{H}_{\nu_{m,n}}(\M[m,n]; \Z)$ is
    torsion-free if and only if $n \ge 2m-4$.
    \label{chesstorsion1-conj}
  \end{conjecture}
  The conjecture is known to be true in all
  cases but 
  $n=2m-4$ and $n=2m-3$;
  {\shareshian} and Wachs \cite{ShWa} settled the case
  $n=2m-2$.
  \begin{corollary}[{\shareshian} \& Wachs \cite{ShWa}]
    For all $(m,n) \neq (1,1)$, we have that
    $\tilde{H}_{\nu_{m,n}}(\M[m,n]; \Z)$
    is nonzero.
    \label{chesstorsion-cor}
  \end{corollary}

  \section{Higher-degree homology}
  \label{generalchess-sec}

  In Section~\ref{higherchess-sec}, we detect $3$-torsion in
  higher-degree homology groups of $\M[m,n]$. 
  In Section~\ref{boundschess-sec}, we proceed with
  upper bounds on the dimension of the homology over $\Z_3$.

  \subsection{$3$-torsion in higher-degree homology groups}
  \label{higherchess-sec}

  This section builds on work previously published in the author's
  thesis \cite{thesis,lmnthesis}. 
  Fix $n,d \ge 0$ and
  let $\gamma$ be an element in $\tilde{H}_{d-1}(\M[n];\Z)$; note
  that we consider the matching complex $\M[n]$.
  For each $k \ge 0$, define a map
  \[
  \left\{
  \begin{array}{l}
    \theta_k : \tilde{H}_{k-1}(\M[k,k+1];\Z)
    \rightarrow \tilde{H}_{k-1+d}(\M[2k+1+n];\Z) \\
    \theta_k(z) =  z \wedge \gamma^{(2k+1)},
  \end{array}
  \right.
  \]
  where we obtain $\gamma^{(2k+1)}$ from
  $\gamma$ by replacing each occurrence of the vertex $i$ with 
  $i+2k+1$ for every $i \in [n]$. 

  For any prime $p$, we have that
  $\theta_k$ induces a homomorphism
  \[
  \theta_{k} \tensor_\Z \iota_p : \tilde{H}_{k-1}(\M[k,k+1];\Z)
  \tensor_\Z \Z_p \rightarrow
  \tilde{H}_{k-1+d}(\M[2k+1+n];\Z) \tensor_\Z
  \Z_p,
  \]
  where $\iota_p : \Z_p \rightarrow \Z_p$ is the identity.
  The following result about the matching complex is a special case of
  a more general result from a previous paper \cite{bettimatch}.
  \begin{theorem}[Jonsson \cite{bettimatch}]
    Fix $k_0 \ge 0$.
    With notation and assumptions as above, 
    if $\theta_{k_0} \tensor_\Z \iota_p$ is a monomorphism, 
    then $\theta_{k} \tensor_\Z \iota_p$ is a monomorphism for each
    $k \ge k_0$.
    \label{torsionallover-thm}
  \end{theorem}

  As alluded to
  in the proof of Theorem~\ref{m55-thm} in
  Section~\ref{bottomchess-sec}, we have that 
  $\M[k,k+1]$ is an orientable pseudomanifold of dimension $k-1$;
  hence $\tilde{H}_{k-1}(\M[k,k+1];\Z) \cong \Z$.
  {\shareshian} and Wachs \cite[\S6]{ShWa} observed that this group is
  generated by the cycle
  \[
  z_{k,k+1} = \sum_{\pi \in \Symm{[k+1]}}
  \sgn(\pi) \cdot 1\n{\pi(1)}\wedge \cdots \wedge
  k\n{\pi(k)}.
  \]
  Note that the sum is over all permutations on $k+1$ elements.
  Theorem~\ref{torsionallover-thm} implies the following
  result.
  \begin{corollary}
    With notation and assumptions as in
    Theorem~{\rm\ref{torsionallover-thm}}, if 
    $(z_{k_0,k_0+1} \wedge  \gamma^{(2k_0+1)}) \tensor 1$ is nonzero
    in $\tilde{H}_{k_0-1+d}(\M[2k_0+1+n];\Z) \tensor \Z_p$, 
    then
    $(z_{k,k+1} \wedge  \gamma^{(2k+1)}) \tensor 1$ is nonzero
    in $\tilde{H}_{k-1+d}(\M[2k+1+n];\Z) \tensor \Z_p$
    for all $k \ge k_0$. 
    \label{torsionallover-cor}
  \end{corollary}

  We will also need a result about the bottom nonvanishing homology
  of the matching complex. Define
  \begin{eqnarray}
    \nonumber
    \gamma_{3r} &=& (12-23) \wedge (45-56) \wedge (78-89) \\
    & & \wedge \cdots \wedge ((3r-2)(3r-1)-(3r-1)(3r));
    \label{gamman-eq}
  \end{eqnarray}
  this is a cycle in both $\tilde{C}_{r-1}(\M[3r]; \Z)$ 
  and $\tilde{C}_{r-1}(\M[3r+1]; \Z)$.
  \begin{theorem}[Bouc \cite{Bouc}]
    For $r \ge 2$, we have that
    $\tilde{H}_{r-1}(\M[3r+1];\Z) \cong \Z_3$.
    Moreover, this group is generated by 
    $\gamma_{3r}$ and hence by any element obtained from
    $\gamma_{3r}$ by permuting the underlying vertex set.
    \label{bouctor-thm}
  \end{theorem}

  Assume that $m+n \equiv 0 \pmod{3}$ and $m \le n \le 2m$.
  Define the cycle $\gamma_{m,n}$ in
  $\tilde{H}_{\nu_{m,n}}(\M[m,n]; \Z)$ recursively as
  follows, the base case being $\gamma_{1,2} = 1\n{1}-1\n{2}$:
  \begin{equation}
  \gamma_{m,n} =
  \left\{
  \begin{array}{ll}
    \gamma_{m-1,n-2} \wedge (m(\n{n-1})-m\n{n}) & \mbox{if } m<n; \\
    \gamma_{m-2,n-1} \wedge ((m-1)\n{n}-m\n{n}) & \mbox{if } m=n.
  \end{array}
  \right.
  \label{gammamn-eq}
  \end{equation}
  For $n > m$, we define $\gamma_{n,m}$ by
  replacing $i\n{j}$ with 
  $j\n{i}$ in $\gamma_{m,n}$ for each $i \in [m]$ and $j \in [n]$.

  Recall that $\nu_{m,n} = \frac{m+n-4}{3}$ whenever
  $m \le n \le 2m-2$.
  \begin{theorem}
    There is $3$-torsion
    in $\tilde{H}_{d}(\M[m,n]; \Z)$ whenever
    \[
    \left\{
    \begin{array}{ccl}
      m+1 \le n \le 2m-5 \\ \\
      \left\lceil\frac{m+n-4}{3}\right\rceil \le d \le m-3
    \end{array}
    \right.
    \Longleftrightarrow
    \left\{
    \begin{array}{ccl}
      k &\ge& 0\\
      a &\ge& 1\\
      b &\ge& 2,
    \end{array}
    \right.
    \]
    where $k$, $a$, and $b$ are defined as in 
    {\rm(\ref{mnd2kab-eq})}. 
    Moreover, there is $3$-torsion
    in $\tilde{H}_{d}(\M[m,m]; \Z)$ whenever
    \[
    \left\lceil\frac{2m-4}{3}\right\rceil \le d \le m-4
    \Longleftrightarrow
    \left\{
    \begin{array}{ccl}
      k &\ge& 0\\
      a &=& 0 \\
      b &\ge& 3.
    \end{array}
    \right.
    \]
    \label{chessallover-thm}
  \end{theorem}
  \begin{proof}
    Assume that $k \ge 0$, $a \ge  1$, and $b \ge 2$.
    Writing $m_0 = a+3b-2$ and 
    $n_0 = 2a+3b-3$, we have the inequalities
    \begin{equation}
      a+3b-2 \le 2a+3b-3 \le 2a+6b-9
      \Longleftrightarrow 
      m_0 \le n_0 \le 2m_0-5.
    \label{ab-eq}
    \end{equation}

    Note that $m_0+n_0 = 3a+6b-5 \equiv 1 \pmod{3}$.
    Define
    \[
    w_{k+1} = z_{k+1,k+2} \wedge \gamma_{m_0,n_0-1}^{(k+1,k+2)},
    \]
    where we obtain $\gamma_{m_0,n_0-1}^{(k+1,k+2)}$ from 
    the cycle $\gamma_{m_0,n_0-1}$ defined in (\ref{gammamn-eq}) by
    replacing $i\n{j}$ with
    $(i+k+1)(\n{j+k+2})$. View $\gamma_{m_0,n_0-1}$ as an element
    in the homology of $\M[m_0,n_0]$.
    Since $z_{k+1,k+2}$ has type
    $\brackom{k+1,k+2}{k+1}$ and since $\gamma_{m_0,n_0-1}$ has type
    $\brackom{a+3b-2,2a+3b-3}{a+2b-2}$ (or rather
    $\brackom{a+3b-2,2a+3b-4}{a+2b-2} \wedge \brackom{0,1}{0}$),
    we obtain that $w_{k+1}$ has type
    \[
    \brackom{k+1+a+3b-2,k+2+2a+3b-3}{k+1+a+2b-2} =
    \brackom{m,n}{d+1};
    \]
    hence we may view $w_{k+1}$ as an element in
    $\tilde{H}_d(\M[m,n];\Z)$. 

    Choosing $k=0$, we obtain that 
    \[
    w_1 = z_{1,2} \wedge \gamma_{m_0,n_0-1}^{(1,2)}. 
    \]
    We claim that $w_1$ has order three when viewed as an element in
    \[
    \tilde{H}_{\frac{m_0+n_0-1}{3}}(\M[m_0+n_0+3];\Z) = 
    \tilde{H}_{a+2b-2}(\M[3a+6b-2];\Z).
    \]
    Namely, we may relabel the vertices to transform
    $w_1$ into the
    cycle $\gamma_{m_0+n_0+2}$ defined in (\ref{gamman-eq}).
    Since $m_0 + n_0 + 3 \ge 13$, Theorem~\ref{bouctor-thm}
    yields the claim.

    Applying Corollary~\ref{torsionallover-cor}, we 
    conclude that $w_{k+1} \tensor 1$ is a nonzero element 
    in the group $\tilde{H}_{k+a+2b-2}(\M[2k+3a+6b-2];\Z) \tensor
    \Z_3 = \tilde{H}_{d}(\M[m+n];\Z) \tensor
    \Z_3$ for every $k \ge 0$. 
    As a consequence, $w_{k+1} \tensor 1$ is nonzero  
    also in 
    \[
    \tilde{H}_{k+a+2b-2}(\M[k+a+3b-1,k+2a+3b-1];\Z) \tensor
    \Z_3 = \tilde{H}_d(\M[m,n];\Z) \tensor \Z_3  
    \]
    for every $k \ge 1$. Since $\tilde{H}_{a+b-3}(\M[m_0,n_0];\Z)$ is
    an elementary $3$-group by 
    Theorem~\ref{chesstorsion-thm} and 
    (\ref{ab-eq}), the order of
    $\gamma_{m_0,n_0-1}$ in $\tilde{H}_{r}(\M[m_0,n_0];\Z)$
    is three. It follows that the order of $w_{k+1}$ in
    $\tilde{H}_d(\M[m,n];\Z)$ is three as well.

    The remaining case is $m=n$, in which case the upper bound on 
    $d$ is $m-4$ rather than $m-3$. Since $a=0$, we get
    \[
    \left\{
    \begin{array}{rcrcrcrcr}
      k & = & -2m &  + & 3d & + & 4 \\
      b & = &   m &  - &  d & - & 1
    \end{array}
    \right. 
    \Leftrightarrow
    \left\{
    \begin{array}{rcrcrcrcl}
      m &  = & k & + & 3 b & - & 1 \\
      d &  = & k & + & 2 b & - & 2.
    \end{array}
    \right. 
    \]
    Clearly, $k \ge 0$ and $b \ge 3$.

    Consider the cycle $w_{k+1} = z_{k+1,k+2} \wedge
    \gamma_{3b-2,3b-4}^{(k+1,k+2)}$. 
    By Corollary~\ref{torsionallover-cor}, $w_{k+1} \tensor 1$ is
    nonzero in
    $\tilde{H}_{k+2b-2}(\M[2k+6b-2];\Z) \tensor \Z_3$. Namely, up to
    the names of the vertices, $w_1$ coincides with
    $\gamma_{6b-3}$  in (\ref{gamman-eq}), which is a
    nonzero element of order three in the group
    $\tilde{H}_{2b-2}(\M[6b-2];\Z)$ by
    Theorem~\ref{bouctor-thm}; $b \ge 3$.
    We conclude that $w_{k+1} \tensor 1$ is
    a nonzero element in $\tilde{H}_{k+2b-2}(\M[k+3b-1,k+3b-1];\Z)
    \tensor \Z_3 = \tilde{H}_d(\M[m,m];\Z) \tensor \Z_3$. Since 
    $3b-3 \ge 6$, we have that
    $\gamma_{3b-2,3b-4}$ must have order three in
    $\tilde{H}_{2b-3}(\M[3b-2,3b-3];\Z)$; apply
    Theorem~\ref{chesstorsion-thm}. This implies that the same must
    be true for $w_{k+1}$ in $\tilde{H}_d(\M[m,m];\Z)$.
  \end{proof}

  \begin{corollary}
    The group $\tilde{H}_{5}(\M[8,9]; \Z) =
    \tilde{H}_{\nu_{8,9}}(\M[8,9]; \Z)$ contains nonvanishing  
    $3$-torsion. As a consequence, 
    there is nonvanishing $3$-torsion in 
    $\tilde{H}_{\nu_{m,n}}(\M[m,n]; \Z)$ whenever
    $m \le n \le 2m-5$.
  \end{corollary}
  \begin{proof} 
    The first statement is a consequence, of 
    Theorem~\ref{chessallover-thm}; choose $k=2$, $a=1$, and
    $b=2$. For the second statement, apply
    Theorem~\ref{chesstorsion-thm}.
  \end{proof}

  \begin{theorem}
    For $1 \le m \le n$, the group $\tilde{H}_{d}(\M[m,n]; \Z)$ is
    nonzero if and only if either
    \[
    \left\lceil\frac{m+n-4}{3}\right\rceil \le d \le m-2
    \Longleftrightarrow
    \left\{
    \begin{array}{ccl}
      k &\ge& 0\\
      a &\ge& 0\\
      b &\ge& 1
    \end{array}
    \right.
    \]
    or 
    \[
    \left\{
    \begin{array}{ccl}
      m &\ge& 1\\
      n &\ge& m+1\\
      d &=& m-1
    \end{array}
    \right.
    \Longleftrightarrow
    \left\{
    \begin{array}{ccl}
      k &\ge& 2-a\\
      a &\ge& 1\\
      b &=& 0,
    \end{array}
    \right.
    \]
    where $k$, $a$, and $b$ are defined as in {\rm(\ref{mnd2kab-eq})}.
    \label{chesshomology-thm}
  \end{theorem}
  \begin{proof}
    For homology to exist, we certainly must have that $b \ge 0$, 
    and we restrict to $a \ge 0$ by assumption. Moreover, $b = 0$
    means that
    $d = m-1$, in which case there is homology only if $m \le n-1$,
    hence $a \ge 1$ and $k+a \ge 2$; for the latter inequality, recall
    that we restrict our attention to $m \ge 1$. Finally, $k < 0$
    reduces  to the case $b = 0$, because we then have homology 
    only if $n \ge 2m+2$ and $d=m-1$; apply
    Theorem~\ref{smallest-thm}.
    
    For the other direction, Theorem~\ref{chessallover-thm} yields
    that we only need to consider the following cases:

    $\bullet$
      $k \ge 0$, $a=0$, and $b=2$. By Theorem~\ref{chessfinite-thm}, 
      we have infinite homology for $a=0$ and $b=2$ if and only if $k
      \ge (b-1)(a+b-1) = a+1 = 1$. The remaining case is $(k,a,b) =
      (0,0,2) \Longleftrightarrow (m,n,d) = (5,5,2)$, in which case we
      have nonzero homology by Theorem~\ref{m55-thm}.

    $\bullet$
      $k \ge 0$, $a \ge 0$, and $b = 1$. This time,
      Theorem~\ref{chessfinite-thm} yields infinite homology 
      for $a \ge 0$ and $b=1$ as soon as $k \ge 0$.

    $\bullet$
      $k \ge 2-a$, $a \ge 1$, and $b=0$. By yet another application of
      Theorem~\ref{chessfinite-thm}, we have infinite homology for
      $b=0$ whenever $a \ge 1$, $k \ge 1-a$, and $k+a \ge 2$. 
      Since the third inequality implies the second, we are done.
  \end{proof}

  \begin{conjecture}[{\shareshian} \& Wachs \cite{ShWa}]
    For $1 \le m \le n$, the group $\tilde{H}_{d}(\M[m,n]; \Z)$
    contains $3$-torsion if and only if 
    \[
    \left\{
    \begin{array}{ccl}
      m \le n \le 2m-5 \\ \\
      \left\lceil\frac{m+n-4}{3}\right\rceil \le d \le m-3
    \end{array}
    \right.
    \Longleftrightarrow
    \left\{
    \begin{array}{ccl}
      k &\ge& 0\\
      a &\ge& 0\\
      b &\ge& 2.
    \end{array}
    \right.
    \]
    \label{chesstorsion2-conj}
  \end{conjecture}
  Note that Conjecture~\ref{chesstorsion2-conj} implies
  Conjecture~\ref{chesstorsion1-conj}.
  Conjecture~\ref{chesstorsion2-conj}
  remains unsettled in the following cases:
  \begin{itemize}
  \item
    $d=m-2$: $9 \le m+2 \le n \le 2m-3$. Equivalently, $k \ge 1$, $a
    \ge 2$, and $b = 1$. 
    Conjecture: There is no $3$-torsion.
  \item
    $d=m-3$: 
    $8 \le m = n$.
    Equivalently,
    $k \ge 3$, $a = 0$, and $b = 2$.  
    Conjecture: There is $3$-torsion.
  \end{itemize}
  The conjecture is fully settled for $n = m+1$ and $n \ge
  2m-2$; see {\shareshian} and Wachs \cite{ShWa} for the case
  $n=2m-2$, and use Theorem~\ref{smallest-thm} for the case $n \ge
  2m-1$.
  For the case $n=m+1$, we have that $\tilde{H}_{m-2}(\M[m,m+1];\Z)$
  is torsion-free, because $\M[m,m+1]$ is an orientable
  pseudomanifold; see Spanier \cite[Ex. 4.E.2]{Spanier}. 

  \subsection{Bounds on the homology over $\Z_3$}
  \label{boundschess-sec}

  Fix a field $\Field$ and let 
  \begin{eqnarray*}
  \betti{m,n}{d} &=&  \dim_\Field \tilde{H}_d(\M[m,n];\Field);\\
  \abetti{m,n}{d} &=&  \dim_\Field \tilde{H}_d(\Mupp[m,n];\Field);
  \end{eqnarray*}
  $\Mupp[m,n]$ is defined as in (\ref{Gamma-eq}).
  \begin{lemma}
    For each $m \ge 2$ and $n \ge 3$, we have that
    \[
    \betti{m,n}{d} \le
    \betti{m-2,n-1}{d-1} 
    + (m-2) \betti{m-1,n-1}{d-1}
    + 2\mbox{$\binom{n-1}{2}$}\betti{m-2,n-3}{d-2}.
    \]
    Thus, by symmetry,
    \[
    \betti{m,n}{d} \le
    \betti{m-1,n-2}{d-1} + 
    (n-2) \betti{m-1,n-1}{d-1} + 
    2\mbox{$\binom{m-1}{2}$}\betti{m-3,n-2}{d-2} 
    \]
    whenever $m \ge 3$ and $n \ge 2$. 
    \label{dmtchess-lem}
  \end{lemma}
  \begin{proof}
    By the long exact 00-$\Gamma$-11 sequence in
    Section~\ref{exseq00-G-11-sec}, we have that
    \[
    \betti{m,n}{d} \le \abetti{m,n}{d} + (m-2) \betti{m-1,n-1}{d-1}.
    \]
    Moreover, the long exact $\Gamma$-21-23 sequence in
    Section~\ref{exseqG-21-23-sec} yields the inequality
    \[
    \abetti{m,n}{d} \le \betti{m-2,n-1}{d-1} + 
    2\mbox{$\binom{n-1}{2}$} \betti{m-2,n-3}{d-2}.
    \]
    Summing, we obtain the desired inequality.
  \end{proof}
  Define
  $\hat{\beta}_{k}^{a,b} = \betti{m,n}{d}$,
  where $k$, $a$, and $b$ are defined as in (\ref{mnd2kab-eq}). We may
  rewrite the second inequality in Lemma~\ref{dmtchess-lem} as
  follows:
  \begin{corollary}
    We have that
    \begin{equation}
      \hat{\beta}_{k}^{a,b} \le
      \hat{\beta}_{k}^{a-1,b} +  (k+2a+3b-3)\hat{\beta}_{k-1}^{a,b} + 
      2\mbox{$\binom{k+a+3b-2}{2}$}\hat{\beta}_{k-1}^{a+1,b-1}
      \label{hatbetachess-eq}
    \end{equation}
    for $k \ge 0$, $a \ge 0$, and $b \ge 2$.
    \label{dmtchess-cor}
  \end{corollary}

  \begin{theorem}
    With $\Field = \Z_3$ and $d = \nu_{m,n}$, the second bound in
    Lemma~{\rm\ref{dmtchess-lem}} is sharp whenever
    $m \le n \le 2m-5$,
    $m+n \equiv 1 \pmod{3}$, and $(m,n) \neq (5,5)$. 
    Equivalently, the bound is sharp whenever $k = 0$, $a \ge 0$, $b
    \ge 2$, and $(k,a,b) \neq (0,0,2)$, where $k$, $a$, and $b$ are
    defined as in {\rm(\ref{mnd2kab-eq})}.
    \label{basecase-thm}
  \end{theorem}
  \begin{proof}
    Since $\hat{\beta}^{a,b}_0 = 1$ for $a \ge 0$ and $b \ge 2$ by
    Theorem~\ref{chesstorsion-thm}, 
    it suffices to prove that
    \begin{equation}
      \hat{\beta}_{0}^{a-1,b} +  (2a+3b-3)\hat{\beta}_{-1}^{a,b} + 
      2\mbox{$\binom{a+3b-2}{2}$}\hat{\beta}_{-1}^{a+1,b-1} = 1
    \label{basecase-eq}
    \end{equation}
    for all $a$ and $b$ as in the theorem;
    apply Corollary~\ref{dmtchess-cor}.
    Clearly, $\hat{\beta}_{0}^{a-1,b} = 1$; when $a=0$, use the fact
    that $\hat{\beta}_{0}^{-1,b} = \hat{\beta}_{0}^{1,b-1}$. Moreover, 
    Theorem~\ref{smallest-thm} yields that $\hat{\beta}_{-1}^{a,b} =
    0$ whenever $a \ge 0$ and $b \ge 1$. As a consequence, we are done.
  \end{proof}
  
  \begin{theorem}
    For each $k \ge 0$, there is a polynomial $f_k(a,b)$ of degree
    $3k$ such that  
    $\hat{\beta}^{a,b}_{k} \le f_{k}(a,b)$
    whenever $a \ge 0$ and $b \ge k+2$ and such that
    \[
    f_k(a,b) = \frac{1}{3^kk!}\left((a+3b)^3-9b^3\right)^{k} +
    \epsilon_k(a,b) 
    \]
    for some polynomial $\epsilon_k(a,b)$ of degree at most $3k-1$.
    Equivalently, 
    \[
    \betti{m,n}{d} \le f_{3d-m-n+4}(n-m,m-d-1)
    \]
    for $m \le n \le 2m-5$ and 
    $\frac{m+n-4}{3} \le d \le \frac{2m+n-7}{4}$. 
    \label{bettichess-thm}
  \end{theorem}
  \begin{proof}
    The case $k=0$ is a consequence of
    Theorem~\ref{chesstorsion-thm}. 
    Assume that $k \ge 1$ and $b > k+2$.

    First, assume that $a > 0$. 
    Induction and Corollary~\ref{dmtchess-cor} yield that
    \begin{eqnarray*}
      \hat{\beta}_{k}^{a,b} -
      \hat{\beta}_{k}^{a-1,b} &\le&  
      (k+2a+3b-3)f_{k-1}(a,b) \nonumber\\
      &+& 
      2\mbox{$\binom{k+a+3b-2}{2}$} f_{k-1}(a+1,b-1),
      \label{hatbetasimpchess2-eq}
    \end{eqnarray*}
    where $f_{k-1}$ is a polynomial with properties as in the
    theorem. The right-hand side is of the form
    \[
    g_k(a,b) = 
    \frac{1}{3^{k-1}(k-1)!}\left((a+3b)^3-9b^3\right)^{k-1}(a+3b)^2 +
    h_k(a,b),
    \]
    where $h_k(a,b)$ is a polynomial of degree at most $3k-2$.
    Now, 
    \begin{eqnarray*}
      & & 
      \frac{1}{3^{k-1} (k-1)!}\left((a+3b)^3-9b^3\right)^{k-1}(a+3b)^2
      \\
      &=&
      \frac{1}{3^{k-1} (k-1)!} \sum_{\ell=0}^{k-1} \binom{k-1}{l}
      (a+3b)^{3k-3\ell-1} (-9b^3)^{\ell}.
    \end{eqnarray*}
    Summing over $a$, we obtain that
    \[
    \hat{\beta}_{k}^{a,b} \le 
    \hat{\beta}_{k}^{0,b} +  \sum_{i=1}^a g_k(i,b).
    \]
    The right-hand side is a polynomial in $a$ and $b$ with dominating
    term
    \begin{eqnarray}
      & &
      \frac{1}{3^{k-1} (k-1)!} \sum_{\ell=0}^{k-1} \binom{k-1}{l}
      \frac{(a+3b)^{3k-3\ell}-(3b)^{3k-3\ell}}{3k-3\ell}
      (-9b^3)^{\ell} \nonumber \\
      &=&
      \frac{1}{3^{k} k!} \sum_{\ell=0}^{k} \binom{k}{l}
      \left(((a+3b)^3)^{k-\ell} - (27b^3)^{k-\ell}\right)
      (-9b^3)^{\ell}
      \nonumber \\
      &=&
      \frac{1}{3^{k} k!} \left((a+3b)^3-9b^3\right)^k
      - \frac{1}{3^{k} k!}(18b^3)^{k}. \label{tailbterm-eq}
    \end{eqnarray}

    Proceeding with $\hat{\beta}_{k}^{0,b}$ for $b \ge k+3$, note that  
    $\hat{\beta}_{k}^{-1,b} = \hat{\beta}_{k}^{1,b-1}$. As a
    consequence,  
    \begin{eqnarray*}
      \hat{\beta}_{k}^{0,b} &\le&
      \hat{\beta}_{k}^{1,b-1} +  (k+3b-3)\hat{\beta}_{k-1}^{0,b} + 
      2\mbox{$\binom{k+3b-2}{2}$}\hat{\beta}_{k-1}^{1,b-1} \\
      &\le& \hat{\beta}_{k}^{0,b-1} +  
      (k+3b-4)\hat{\beta}_{k-1}^{1,b-1} + 
      2\mbox{$\binom{k+3b-4}{2}$}\hat{\beta}_{k-1}^{2,b-2}\\
      &+& (k+3b-3)\hat{\beta}_{k-1}^{0,b} + 
      2\mbox{$\binom{k+3b-2}{2}$}\hat{\beta}_{k-1}^{1,b-1}.
    \end{eqnarray*}
    Using induction, we conclude that
    \begin{eqnarray*}
    \hat{\beta}_{k}^{0,b} &\le&
    \hat{\beta}_{k}^{0,b-1} +  
    9b^2f_{k-1}(2,b-2) +  
    9b^2f_{k-1}(1,b-1) +O(b^{3k-2}) \\
    &=& 18b^2 \frac{(18b^3)^{k-1}}{3^{k-1}(k-1)!} +O(b^{3k-2})
    = \frac{18^kb^{3k-1}}{3^{k-1}(k-1)!} +O(b^{3k-2}),
    \end{eqnarray*}
    where $f_{k-1}$ is a polynomial with properties as in the
    theorem. Summing over $b$, we may conclude that
    $\hat{\beta}_{k}^{0,b}$ is bounded by a polynomial in $b$ with 
    dominating term $\frac{18^kb^{3k}}{3^{k}k!}$. 
    Combined with (\ref{tailbterm-eq}), this yields the theorem.
  \end{proof}

  \section*{Acknowledgments}

  I thank an anonymous referee for several useful comments.
  This research was carried out at the Technische Universit\"at Berlin
  and at the Massachusetts Institute of Technology in Cambridge, MA.

  \newpage

  \begin{table}
    \caption{The exponent $\epsilon_{m,n}$ of
      $\tilde{H}_{\nu_{m,n}}(\M[m,n]; \Z)$ 
      for $m \le n \le 2m-5$ and $(m,n) \not\in 
      \{(6,6),(7,7),(8,9)\}$.
      On the right we give the values $k$, $a$, and $b$ defined as in
      (\ref{mnd2kab-eq}).}
    \begin{center}
    \begin{tabular}{|c|c|c||c|c|c|}
      \hline
      $2m-n$ & Restriction &  $\epsilon_{m,n}$ &
      $k$ & $a$ & $b$ \\
      \hline
      \hline
      $5$ &  &  $3$ & $0$ & $\ge 0$ & $2$ \\
      \cline{1-5}
      $6$ & $m \ge 7$ &  divides $\epsilon_{7,8}$ 
      & $1$ & $\ge 1$ &  \\
      \cline{1-5}
      $7$   & $m \ge 9$ &  divides $\epsilon_{9,11}$ 
      & $2$ & $\ge 2$ &  \\
      \hline
      $8$ &  &  $3$ 
      & $0$ & $\ge 0$ & $3$\\
      \cline{1-5}
      $9$ &    &  divides $\gcd(9,\epsilon_{9,9})$ 
      & $1$ & $\ge 0$ & \\
      \cline{1-5}
      $10$ & $m=10$     &  multiple of $3$ 
      & $2$ & $=0$ &
      \\
      \cline{2-3} \cline{5-5}
      & $m \ge 11$ &  divides $\epsilon_{7,8}$ 
      & & $\ge 1$ &  \\
      \hline
      \hline
      $11+3t$ & $t \ge 0$ &  $3$ 
      & $0$ & $\ge 0$ & $4+t$ \\
      \cline{1-1}\cline{3-4}
      $12+3t$ &  &  divides $\gcd(9,\epsilon_{9,9})$ 
      & $1$ &  & 
      \\
      \cline{1-1}\cline{4-4}
      $13+3t$ & &  
      & $2$ & & 
      \\
      \hline
    \end{tabular}
    \end{center}
    \label{chessexp-fig}
  \end{table}

\end{document}